\documentclass[11pt]{amsart}

\usepackage{amsmath,amssymb,amsthm,mathtools}
\usepackage{mathrsfs}
\usepackage{graphicx}
\usepackage{booktabs}
\usepackage{tikz-cd}
\usepackage[colorlinks=true,
linkcolor=blue,
citecolor=blue,
urlcolor=blue]{hyperref}
\usepackage[capitalise,noabbrev]{cleveref}

\makeatletter
\renewcommand{\section}{
	\@startsection{section}{1}{\z@}
	{2.5ex plus 1ex minus .2ex}
	{1.2ex plus .2ex}
	{\normalfont\large\bfseries}}
\renewcommand{\subsection}{
	\@startsection{subsection}{2}{\z@}
	{2ex plus .8ex minus .2ex}
	{.8ex plus .2ex}
	{\normalfont\normalsize\bfseries}}
\makeatother

\newtheorem{theorem}{Theorem}[section]
\newtheorem*{theorem*}{Theorem}

\newtheorem{lemma}{Lemma}[section]

\title{Geometric Construction of the McKay-Slodowy Correspondence} 
\author{Shengyu Hou} 

\begin{document}
	
	\begin{abstract}
		Let $G\subseteq \operatorname{SL}_2(\mathbb{C})$ be a finite group and let $H\trianglelefteq G$ be a normal subgroup. The McKay correspondence for $H$ says there is a one-to-one correspondence between nontrivial irreducible representations of $H$ and irreducible components of the exceptional locus of the minimal resolution of $\mathbb{C}^2/H$. We prove that this correspondence for $H$ is $G$-equivariant, and it induces a one-to-one correspondence between induced representations and push-forwards of exceptional curves. We also identify the intersection pairing of curves with the inner product of representations under this correspondence. 
	\end{abstract}
	
	\maketitle
	
	\section{Introduction}
	The McKay correspondence links representations of finite groups with the geometry of surface singularities. Let $H\subseteq \operatorname{SL}_2(\mathbb{C})$ be a finite group. Denote by $\operatorname{Irr}(H)$ the set of isomorphism classes of irreducible finite-dimensional complex representations of $H$. Let $1$ be the trivial representation and set $\operatorname{Irr}_0(H):=\operatorname{Irr}(H)\setminus\{1\}$. Let $T_H$ be the natural representation  $H\hookrightarrow \operatorname{SL}_2(\mathbb{C})$. For each $\tau\in\operatorname{Irr}(H)$, write
	\[
	T_H\otimes\tau
	={\bigoplus}_{\sigma\in\operatorname{Irr}(H)}a_{\sigma\tau}\mkern1mu\sigma.
	\]
	There is a natural $H$-action on the two-dimensional complex affine space $\mathbb{C}^2$. Let $\varphi\colon X\to\mathbb{C}^2/H$ be the minimal resolution and let $\operatorname{Irr}(\operatorname{Exc}(\varphi ))$ be the set of all irreducible components of the exceptional locus. The McKay correspondence says there is a one-to-one correspondence
	\[
	\operatorname{Irr}_0(H)\longrightarrow
	\operatorname{Irr}(\operatorname{Exc}(\varphi)),
	\qquad \tau\longmapsto E_\tau,
	\] and under this correspondence, for all $\tau , \sigma \in \operatorname{Irr}_0(H)$, the intersection number $
	(E_\tau\cdot E_\sigma)_X$ is equal to $a_{\sigma\tau}-2\delta_{\sigma\tau}$, where $\delta _{\sigma \tau }$ is the Kronecker delta.
	Note that $a_{\sigma \tau }$ is equal to the inner product of representations $\langle T_H\otimes \tau , \sigma \rangle  _H:=\dim_\mathbb{C}\operatorname{Hom}_{\mathbb{C}[H]}(T_H\otimes \tau , \sigma )$. Thus the equation can also be written as 
	$$(E_\tau \cdot E_\sigma )_X=\langle (T_H-2)\tau , \sigma \rangle _H.$$
	
	McKay first found this relation in \cite{McK80}. Gonzalez-Sprinberg and Verdier gave a geometric construction of the correspondence in \cite{GSV83}, and Kn\"orrer later gave another proof in \cite{Kn85}. In \cite{IN99}, Ito and Nakamura identified $X$ with the $H$-Hilbert scheme and gave another geometric construction. Kapranov and Vasserot lifted the correspondence to a Fourier--Mukai equivalence of derived categories in \cite{KV00}. Bridgeland, King and Reid developed the derived McKay correspondence into higher dimensional case in \cite{BKR01}.
	
	Now let $H\trianglelefteq G\subseteq\operatorname{SL}_2(\mathbb{C})$ be two finite groups. From now on, $\operatorname{Ind}$ and $\operatorname{Res}$ mean induction and restriction of representations between $H$ and $G$. Set $\operatorname{Ind}\mkern1mu\operatorname{Irr}(H)
	:=\{\operatorname{Ind}\tau\mid \tau\in\operatorname{Irr}(H)\}$ and $\operatorname{Ind}\mkern1mu\operatorname{Irr}_0(H)=\operatorname{Ind}\mkern1mu\operatorname{Irr}(H)\setminus \{\operatorname{Ind}1\}$. In \cite[Appendix~III]{Slo}, 
	Slodowy studied the action of the natural representation $T_G:G\hookrightarrow \operatorname{SL}_2(\mathbb{C})$ on these induced representations. For all $\operatorname{Ind}\tau\in\operatorname{Ind}\mkern1mu\operatorname{Irr}(H)$, consider decompositions of the form
	\[
	T_G\otimes\operatorname{Ind}\tau
	={\bigoplus}_{\operatorname{Ind}\sigma\in\operatorname{Ind}\mkern1mu\operatorname{Irr}(H)}
	b_{\operatorname{Ind}\tau,\operatorname{Ind}\sigma}\mkern1mu\operatorname{Ind}\sigma.
	\]		
	The action of $G$ on $\mathbb{C}^2$ descends to $\mathbb{C}^2/H$ and lifts uniquely to $X$ (Lemma \ref{lem:lift-and-linearization}). We write $\pi\colon X\to X/G$ for the quotient map. Thus we have a commutative diagram
	\[
	\begin{tikzcd}
		X \arrow[d,"\varphi"'] \arrow[r,"\pi"]
		& X/G \arrow[d] \\
		\mathbb{C}^2/H \arrow[r]
		& \mathbb{C}^2/G.
	\end{tikzcd}
	\]
	It is natural to ask whether there is a correspondence between $\operatorname{Ind}\mkern1mu\operatorname{Irr}_0(H)$ and irreducible exceptional curves on $X/G$, and how the intersection numbers on $X/G$ relate to the decomposition numbers $b_{\operatorname{Ind}\tau , \operatorname{Ind}\sigma }$.
	
	In \cite{DHW26}, Du, Wang and the author computed the intersection matrices case by case and found a relation between them and the matrix $(b_{\operatorname{Ind}\tau , \operatorname{Ind}\sigma })$. These computations verify the relation in all cases and suggest that it should have a more conceptual explanation. The aim of this paper is to give a conceptual explanation. 
	
	Our starting point is the similarity between some formulas in representation theory and geometry. The inner product of representations and the intersection pairing of curves satisfy similar formulas arising from adjoint functors (Lemma \ref{lem:Frobenius-reciprocity-Mackey-formula} and Lemma \ref{lem:projection-formula}):
	$$\langle \operatorname{Res}\rho,\tau\rangle_H
	=\langle \rho,\operatorname{Ind}\tau\rangle_G,
	\qquad
	(\pi^*B\cdot E)_X
	=(B\cdot\pi_*E)_{X/G},$$
	while the Mackey formula and the projection formula (Lemma \ref{lem:Frobenius-reciprocity-Mackey-formula} and Lemma \ref{lem:projection-formula}) also have the same form:
	$$\operatorname{Res}\mkern1mu\operatorname{Ind}\tau
	={\bigoplus}_{gH\in G/H}\tau^g,
	\qquad
	\pi^*\pi_*E
	={\sum}_{gH\in G/H}g^*E.$$
	These formulas show that induction and restriction of representations are parallel to push-forward and pull-back of exceptional curves, and the inner product is parallel to the intersection pairing. Together with the $G$-equivariance of the McKay correspondence, this gives a natural correspondence between induced representations and push-forwards of exceptional curves. This is our main result:
	
	\begin{theorem}\label{thm:main-theorem}
		Notation as above. Let $\pi_*\operatorname{Irr}(\operatorname{Exc}(\varphi )):=\{\pi_*E\mid E\in \operatorname{Irr}(\operatorname{Exc}(\varphi ))\}$. Then:
		\begin{enumerate}
			\item The McKay correspondence $\tau\mapsto E_\tau$ is $G$-equivariant. In other words, $g^*E_\tau=E_{\tau^g}$ for every $g\in G$ and $\tau\in\operatorname{Irr}_0(H)$.
			
			\item $\operatorname{Ind}\mkern1mu\operatorname{Irr}_0(H)
			\to
			\pi_*\operatorname{Irr}(\operatorname{Exc}(\varphi))$, $\operatorname{Ind}\tau\mapsto\pi_*E_\tau$
			is a one-to-one correspondence.
			
			\item For all $\tau,\sigma\in\operatorname{Irr}_0(H)$, 
			\[
			(\pi_*E_\tau\cdot\pi_*E_\sigma)_{X/G}
			=\big\langle
			(T_G-2)\otimes\operatorname{Ind}\tau,
			\operatorname{Ind}\sigma
			\big\rangle_G.
			\]
		\end{enumerate}
	\end{theorem}
	
	Part~(1) follows from the natural $G$-linearization of the vector bundle in the Gonzalez-Sprinberg--Verdier construction. It gives part~(2) by passing to $G$-orbits. Part~(3) then follows from the parallel formulas above.
	
	The paper is organized as follows. In Section~2, we recall basic facts about representations, finite quotients, equivariant sheaves, and intersection pairings on normal surfaces. In Section~3, we recall the geometric construction of the McKay correspondence due to Gonzalez-Sprinberg and Verdier. In Section~4, we prove the main theorem. 
	
	\subsection*{Acknowledgements}
	The author is grateful to Professor Lizhong Wang for suggesting this problem, and is grateful to Professor Yifei Chen for his guidance and support. The author was partially supported by the NSFC Grant No.~12271384.
	
	\section{Preliminaries}
	
	Throughout the paper, the ground field is $\mathbb{C}$. A variety is an integral separated scheme of finite type, a curve is a variety of dimension $1$, and a surface is a variety of dimension $2$.
	
	\subsection{Representations and inner products}\label{subsec:representations}
	
	Let $H$ be a finite group. Let $\operatorname{Irr}(H)$ be the set of all irreducible representations. Its representation group is
	$R(H):=\bigoplus_{\tau \in\operatorname{Irr}(H)}\mathbb{Z}\cdot\tau $. There is an \textbf{inner product} $$R(H)\times R(H)\to \mathbb{Z},\quad \langle \tau , \sigma \rangle_H :=\dim_\mathbb{C}\operatorname{Hom}_{\mathbb{C}[H]}(\tau , \sigma ).$$
	
	Now let $H\trianglelefteq G\subseteq\operatorname{SL}_2(\mathbb{C})$ be two finite groups. The restriction $\operatorname{Res}$ and induction $\operatorname{Ind}$ of representations between $G$ and $H$ are classical in representation theory; see \cite[Chapter XVIII]{Lang02}. For $g\in G$ and an $H$-representation $\tau$, define $\tau^g$ by $\tau^g(h):=\tau(ghg^{-1})$. We have the Frobenius reciprocity and the Mackey formula:
	
	\begin{lemma}[{\cite[Chapter XVIII]{Lang02}}]\label{lem:Frobenius-reciprocity-Mackey-formula}\ 
		\begin{enumerate}
			\item (Frobenius Reciprocity) For all $\rho\in R(G)$ and $\tau \in R(H)$, $$\left\langle \operatorname{Res}\rho, \tau \right\rangle _H=\left\langle \rho, \operatorname{Ind}\tau \right\rangle _G.$$
			\item (Mackey Formula) For all $\tau \in R(H)$,  $$\operatorname{Res}\mkern1mu\operatorname{Ind}\tau ={\bigoplus} _{gH\in G/H}\ \tau ^g.$$
		\end{enumerate}
	\end{lemma}
	
	\subsection{Quotients by finite groups}
	
	Let $\mathcal{C}$ be a category, and let $X\in \mathcal{C}$ be an object with automorphism group $\operatorname{Aut}(X)$. We say a finite group $G$ \textbf{acts} on $X$ if there exists a group homomorphism $G \to \operatorname{Aut}(X)$. Under this homomorphism, each $g \in G$ is identified with an automorphism $g:X \to X$. A morphism $f: X \to Y$ is called \textbf{$G$-invariant} if $G$ acts on $X$ and $f \circ g = f$ for all $g \in G$. A morphism $f: X \to Y$ is called \textbf{$G$-equivariant} if $G$ acts on both $X$ and $Y$, and $f \circ g = g \circ f$ for all $g \in G$. A morphism $\pi: X \to X/G$ is called the \textbf{(categorical) quotient} of $X$ by $G$ if for every $G$-invariant morphism $f: X \to Y$, there exists a unique morphism $\bar{f}: X/G \to Y$ such that $f = \bar{f} \circ \pi$. It follows immediately that the quotient is unique (up to isomorphism) if it exists.
	
	In the category of quasi-projective varieties, \cite[Chapter~3, Exercise~3.23 and Chapter~4, Exercise~1.10]{Liu02} shows that if a finite group $G$ acts on a quasi-projective variety $X$, then the quotient morphism $\pi: X \to X/G$ always exists and is finite and surjective. Moreover, $X/G$ is normal if $X$ is normal. 
	
	\subsection{Equivariant sheaves}
	Let $G$ be a finite group acting on a variety $X$, and let $\mathcal{F}$ be a coherent sheaf on $X$. A \textbf{$G$-linearization} of $\mathcal{F}$ is a family of isomorphisms $\{\lambda _g:g^*\mathcal{F}\xrightarrow{\sim }\mathcal{F}\}_{g\in G}$ such that $\lambda _{gh}=\lambda _h\circ h^*\lambda _g$ for all $g, h\in G$. A \textbf{$G$-sheaf} is a coherent sheaf together with a $G$-linearization. Their category is denoted by $\operatorname{Coh}_G(X)$. We refer to \cite[\S3]{BKR01} for more details on $G$-sheaves.
	
	Suppose $G$ acts trivially on $X$. Then every $G$-sheaf is an $\mathscr{O}_X[G]$-module. Let $\rho$ be a representation of $G$. We have two functors:
	$$
	\begin{array}{ccccccc}
		\operatorname{Coh}(X) & \to & \operatorname{Coh}_G(X), & & \mathcal{F} & \mapsto & \mathcal{F} \otimes_{\mathbb{C}} \rho, \\
		\operatorname{Coh}_G(X) & \to & \operatorname{Coh}(X), & & \mathcal{G} & \mapsto & \operatorname{Hom}_{\mathbb{C}[G]}(\rho, \mathcal{G}).
	\end{array}
	$$
	Since $\mathbb{C}[G]$ is a semisimple ring, these two functors are both exact. In particular, when $\rho = 1$ is the trivial representation, $\mathcal{F} \otimes_{\mathbb{C}} 1$ is the $G$-sheaf with trivial $G$-linearization, and $\operatorname{Hom}_{\mathbb{C}[G]}(1, \mathcal{G})$ is the subsheaf of $G$-invariant sections $\mathcal{G}^G\subseteq \mathcal{G}$.
	
	\subsection{Intersection Pairings and Projection Formulas}
	\label{subsec:Intersection-Pairing-and-Projection-Formulas}
	
	Let $X$ be a variety, and let $Z_1(X)$ be the free abelian group generated by the complete irreducible curves on $X$. For each complete irreducible curve $E\subseteq X$, let $\nu_E\colon \widetilde{E}\to E$ be its normalization. We define the intersection number $\operatorname{Pic}(X)\times Z_1(X)\to\mathbb{Z}$ by $
	(\mathcal{L}\cdot E)_X:=\deg\bigl(\nu_E^*(\mathcal{L}|_E)\bigr)$.
	
	Now suppose that $X$ is a normal quasi-projective surface. We use the intersection theory extended by Mumford, following the conventions and results collected in \cite[Section~1.3]{Nak23}. It gives a symmetric bilinear pairing $(-\cdot -)_X\colon Z_1(X)\times Z_1(X)\to\mathbb{Q}$. If $E_1\in Z_1(X)$ is a Cartier divisor and $\mathscr{O}_X(E_1)$ is the line bundle associated with $E_1$, then $(E_1\cdot E_2)_X=\bigl(\mathscr{O}_X(E_1)\cdot E_2\bigr)_X$.
	
	Let $G$ be a finite group acting on $X$, and let $\pi\colon X\to Y:=X/G$ be the quotient morphism. Let $E\subseteq X$ be an irreducible complete curve and set $B:=\pi(E)$. The \textbf{push-forward} of $E$ is $\pi_*E:=f_EB$, where $f_E:=\deg(\pi|_E\colon E\to B)$. Let $B\subseteq X/G$ be an irreducible complete curve. Its \textbf{pull-back} is $\pi^*B:=\sum_{E\subseteq \pi^{-1}(B)}e_EE$, where $E$ runs through the irreducible components of $\pi^{-1}(B)$ and $e_E$ is defined by $\mathfrak{m}_{\eta_B}\mathscr{O}_{X,\eta_E}=\mathfrak{m}_{\eta_E}^{e_E}$.
	
	For an irreducible complete curve $E\subseteq X$ and $g\in G$, we view $g$ as an automorphism of $X$ and write $g^*E$ for the preimage $g^{-1}(E)$. We have the following two projection formulas.
	
	\begin{lemma}[Projection Formulas]\label{lem:projection-formula}\ 
		\begin{enumerate}
			\item For all $E\in Z_1(X)$ and $B\in Z_1(X/G)$,
			$$
			(\pi^*B\cdot E)_X=(B\cdot\pi_*E)_{X/G}.
			$$
			
			\item For all $E\in Z_1(X)$,
			$$
			\pi^*\pi_*E={\sum}_{gH\in G/H}g^*E,
			$$
			where $H=\ker(G\to\operatorname{Aut}(X))$.
		\end{enumerate}
	\end{lemma}
	\begin{proof}
		Part~(1) follows from \cite[Remark~1.24(2)]{Nak23}. We prove part~(2) using \cite[Remark~1.24(3)]{Nak23}. Let $E\subseteq X$ be an irreducible complete curve. We have $\operatorname{Supp}(\pi^*\pi_*E)=\bigcup_{gH\in G/H}g^*E$. The group $G$ acts transitively on the irreducible components of this support. Hence all components have the same coefficient. Thus $\pi^*\pi_*E=\lambda\sum_{gH\in G/H}g^*E$ for some $\lambda\in\mathbb{Q}$. Pushing it forward gives $\pi_*\pi^*\pi_*E=\lambda\pi_*\sum_{gH\in G/H}g^*E=\lambda|G/H|\pi_*E$. On the other hand, \cite[Remark~1.24(3)]{Nak23} shows $\pi_*\pi^*(\pi_*E)=|G/H|\pi_*E$. Therefore $\lambda=1$.
	\end{proof}
	
	\section{The Gonzalez-Sprinberg--Verdier construction}\label{sec:gsv}
	
	We recall the geometric construction of the McKay correspondence in \cite{GSV83}. Let $H\subseteq\operatorname{SL}_2(\mathbb{C})$ be a finite group. $H$ acts naturally on $\mathbb{C}^2$. Let $\varphi\colon X\to\mathbb{C}^2/H$ be the minimal resolution. Let
	$\mathcal U:=(X\times_{\mathbb{C}^2/H}\mathbb{C}^2)_{\mathrm{red}}$ be the reduced scheme associated to the fiber product of $X$ and $\mathbb{C}^2$ over $\mathbb{C}^2/H$.
	There is a commutative diagram
	\[
	\begin{tikzcd}
		\mathcal U \arrow[d,"q"'] \arrow[r,"p"]
		& X \arrow[d,"\varphi"] \\
		\mathbb{C}^2 \arrow[r]
		& \mathbb{C}^2/H.
	\end{tikzcd}
	\]
	By \cite[Proposition~2.4]{GSV83}, the morphism $p$ is finite and flat, and
	$\mathcal E:=p_*\mathscr O_{\mathcal U}$ is locally free.
	Indeed, if we regard $X$ as the $H$-Hilbert scheme $\operatorname{Hilb}^H(\mathbb{C}^2)$ (see \cite[Theorem 9.3]{IN99}), then $\mathcal{U}$ is its universal family and $\mathcal{E}$ is its universal bundle. Since we do not require the $H$-Hilbert scheme structure of $X$, we use the simpler language above for clarity. 
	
	Let $T_H:H\hookrightarrow \operatorname{SL}_2(\mathbb{C})$ be the natural representation. For all $\tau \in \operatorname{Irr}_0(H)$, let $\mathcal{E}_\tau:=\operatorname{Hom}_{\mathbb{C}[H]}(\tau,\mathcal{E})$. \cite[2.2 Th\'eor\`em (i)]{GSV83} shows the correspondence between $\operatorname{Irr}_0(H)$ and $\operatorname{Irr}(\operatorname{Exc}(\varphi ))$:
	
	\begin{theorem}[{\cite[2.2 Th\'eor\`em (i)]{GSV83}}]\label{thm:GSV-correspondence}\ 
		
		There is a one-to-one correspondence $$\operatorname{Irr}_0(H) \to \operatorname{Irr}(\operatorname{Exc}(\varphi)),\quad\tau \mapsto E_\tau,$$ where $E_\tau $ is characterized by
		$(\det\mathcal{E}_{\sigma })\cdot E_{\tau } =\delta _{\sigma \tau }=
		\begin{cases}
			1 & \sigma =\tau \\
			0 & \sigma \neq \tau 
		\end{cases}$ for all  $\sigma , \tau \in \operatorname{Irr}_0(H)$.
		
		Moreover, for all $\tau , \sigma \in \operatorname{Irr}_0(H)$, $$(E_\tau \cdot E_\sigma)_X =\langle (T_H-2)\tau , \sigma \rangle_H.$$
	\end{theorem}
	
	\section{Proof of the main theorem}
	
	Keep the notation above and let $H\trianglelefteq G\subseteq\operatorname{SL}_2(\mathbb{C})$ be two finite groups. We first prove a lemma:
	
	\begin{lemma}\label{lem:lift-and-linearization}
		The $G$-action on $\mathbb{C}^2/H$ lifts naturally to $X$, and $\mathcal{E}\in \operatorname{Coh}_{G}(X)$. 
	\end{lemma}
	\begin{proof}
		For all $g\in G$, view $g$ as an automorphism $g_{\mathbb{C}^2/H}\colon\mathbb{C}^2/H\to \mathbb{C}^2/H$. The composition $g_{\mathbb{C}^2/H}\circ \varphi :X\to \mathbb{C}^2/H$ is again a minimal resolution of $\mathbb{C}^2/H$. The uniqueness of the minimal resolution of a surface singularity gives a unique automorphism $g_X:X\to X$ such that $\varphi \circ g_X=g_{\mathbb{C}^2/H}\circ \varphi $, and the uniqueness also implies $(g_1g_2)_X=(g_1)_X\circ(g_2)_X$ for all $g_1, g_2\in G$. Thus $G$ acts on $X$. 
		
		Now $G$ acts diagonally on $\mathcal{U}=(X\times _{\mathbb{C}^2/H} \mathbb{C}^2)_\text{red}$ by $g:(x, z)\mapsto (gx, gz)$. Under this action, $p:\mathcal{U}\to X,\ (x, z)\mapsto x$ is $G$-equivariant. Thus for all $g\in G$, the $G$-linearization of $\mathscr{O}_\mathcal{U}$ gives the $G$-linearization of $\mathcal{E}$:
		$$g^*\mathcal{E}=g^*p_*\mathscr{O}_\mathcal{U}= p_*g^*\mathscr{O}_\mathcal{U}\xrightarrow{\sim } p_*\mathscr{O}_\mathcal{U}=\mathcal{E}.$$
	\end{proof}
	
	Let $\pi:X\to X/G$ be the quotient. We get a diagram:
	$$\begin{tikzcd}
		\mathcal{U} \arrow[d, "q"'] \arrow[r, "p"] & X \arrow[d, "\varphi "] \arrow[r, "\pi"] & X/G \arrow[d]  \\
		\mathbb{C}^2 \arrow[r]                     & \mathbb{C}^2/H \arrow[r]                 & \mathbb{C}^2/G
	\end{tikzcd}$$
	
	Recall that $G$ acts on $\operatorname{Irr}_0(H)$ by $\tau \mapsto \tau ^g$ (\S\ref{subsec:representations}) and acts on $\operatorname{Irr}(\operatorname{Exc}(\varphi ))$ by $E\mapsto g^*E$ (\S\ref{subsec:Intersection-Pairing-and-Projection-Formulas}). We prove part~(1) of Theorem~\ref{thm:main-theorem}:
	
	\begin{theorem}\label{thm:equivariant}
		The McKay correspondence $\tau \mapsto E_\tau $ in Theorem~\ref{thm:GSV-correspondence} is $G$-equivariant. In other words, for all $g\in G$ and $\tau\in\operatorname{Irr}_0(H)$
		$$g^*E_\tau=E_{\tau^g}.$$
	\end{theorem}
	\begin{proof}
		For all $g\in G$ and $\sigma \in \operatorname{Irr}_0(H)$, the $G$-linearization of $\mathcal{E}$ gives an isomorphism 
		$$
		g^*\mathcal{E}_\sigma  =g^*\operatorname{Hom}_{\mathbb{C}[H]}(\sigma , \mathcal{E}) = \operatorname{Hom}_{\mathbb{C}[H]}(\sigma ^g, g^*\mathcal{E})\xrightarrow{\sim} \operatorname{Hom}_{\mathbb{C}[H]}(\sigma  ^g, \mathcal{E})=\mathcal{E}_{\sigma ^g}.
		$$
		Now
		$$
		(\det\mathcal E_\sigma)\cdot g^*E_\tau=
		\bigl((g^{-1})^*\det\mathcal E_\sigma\bigr)\cdot E_\tau =
		(\det\mathcal E_{\sigma^{g^{-1}}})\cdot E_\tau =
		\delta_{\sigma^{g^{-1}},\tau}=
		\delta_{\sigma,\tau^g}.
		$$
		Therefore, the characterization of $E_{\tau ^g}$ in Theorem~\ref{thm:GSV-correspondence} shows
		$g^*E_\tau=E_{\tau^g}$. 
	\end{proof}
	
	The $G$-equivariance allows us to descend the McKay correspondence to a correspondence of $G$-orbits. Induction $\operatorname{Ind}^G_H$ detects the orbits of $\operatorname{Irr}_0(H)$, while the push-forward $\pi_*$ detects the orbits of $\operatorname{Irr}(\operatorname{Exc}(\varphi ))$. Set $\operatorname{Ind}\mkern1mu\operatorname{Irr}_0(H):=\{\operatorname{Ind}^G_H\tau\mid \tau \in \operatorname{Irr}_0(H)\}$ and $\pi_*\operatorname{Irr}(\operatorname{Exc}(\varphi ))=\{\pi_*E\mid E\in \operatorname{Irr}(\operatorname{Exc}(\varphi ))\}$. We now prove part~(2) of Theorem~\ref{thm:main-theorem}:
	
	\begin{theorem}\label{thm:our-correspondence}
		The McKay correspondence  $\tau \mapsto E_\tau$ in Theorem \ref{thm:GSV-correspondence} descends to a one-to-one correspondence 
		$$
		\operatorname{Ind} \mkern1mu \operatorname{Irr}_0(H) \to \pi_* \operatorname{Irr}(\operatorname{Exc}(\varphi)), \quad \operatorname{Ind} \mkern1mu \tau \mapsto \pi_*E_\tau.
		$$
	\end{theorem}
	\begin{proof}
		If $\tau , \sigma \in \operatorname{Irr}_0(H)$ satisfy $\operatorname{Ind}\tau =\operatorname{Ind}\sigma $, then $\operatorname{Res}\mkern1mu\operatorname{Ind}\tau =\operatorname{Res}\mkern1mu \operatorname{Ind}\sigma $. The Mackey formula (Lemma \ref{lem:Frobenius-reciprocity-Mackey-formula}) shows that $\tau , \sigma $ lie in the same $G$-orbit. From Theorem \ref{thm:equivariant}, $E_\tau $ and $E_\sigma $ lie in  the same $G$-orbit. Thus $\pi_*E_\tau =\pi_*E_\sigma $. Therefore $\operatorname{Ind}\tau \mapsto \pi_*E_\tau$ is well-defined. 
		
		If $E_\tau , E_\sigma \in \operatorname{Irr}(\operatorname{Exc}(\varphi ))$ satisfy $\pi_*E_\tau =\pi_*E_\sigma $, then $\pi^*\pi_*E_\tau =\pi^*\pi_*E_\sigma $. The projection formula (Lemma \ref{lem:projection-formula}) shows that $E_\tau $ and $E_\sigma $ lie in the same $G$-orbit. From Theorem \ref{thm:equivariant}, $\tau $ and $\sigma $ lie in  the same $G$-orbit. Thus $\operatorname{Ind}\tau =\operatorname{Ind}\sigma $. Therefore $\operatorname{Ind}\tau \mapsto \pi_*E_\tau$ is injective. 
		
		Since $\tau \mapsto E_\tau$ is surjective, $\operatorname{Ind}\tau \mapsto \pi_*E_\tau$ is also surjective.
	\end{proof}
	
	It remains to compare the intersection pairing of exceptional curves and the inner products of representations. Let $T_G:G\hookrightarrow \operatorname{SL}_2(\mathbb{C})$ be the natural representation. We now prove part~(3) of Theorem~\ref{thm:main-theorem}:
	
	\begin{theorem}\label{thm:pairing-formula}
		The correspondence $\operatorname{Ind}\tau \mapsto \pi_*E_\tau $ in Theorem \ref{thm:our-correspondence} satisfies: for all $\operatorname{Ind} \tau , \operatorname{Ind} \sigma  \in \operatorname{Ind} \mkern1mu \operatorname{Irr}_0(H)$,
		$$
		(\pi_*E_\tau  \cdot \pi_*E_\sigma)_{X/G}  = \left\langle (T_G-2)\otimes \operatorname{Ind}\tau , \operatorname{Ind}\sigma \right\rangle_G.
		$$
	\end{theorem}
	\begin{proof}
		By Theorem \ref{thm:GSV-correspondence}, $(E_{\tau^g}\cdot E_\sigma)_X=\langle (T_H-2)\tau^g,\sigma\rangle_H$. By Lemma \ref{lem:Frobenius-reciprocity-Mackey-formula} and Lemma \ref{lem:projection-formula}, we have:
		\[\begin{split}
			(\pi_*E_\tau  \cdot \pi_*E_\sigma)_{X/G} &= (\pi^*\pi_*E_\tau  \cdot E_\sigma )_X \\
			&={\sum}_{gH \in G/H} (g^*E_\tau  \cdot E_\sigma )_X \\
			&= {\sum}_{gH \in G/H} (E_{\tau^g}  \cdot E_\sigma )_X \\
			&= {\sum}_{gH \in G/H} \langle (T_H-2)\otimes \tau ^g, \sigma  \rangle_H \\
			&= \langle (T_H-2)\otimes \operatorname{Res}\mkern1mu\operatorname{Ind}\tau , \sigma  \rangle_H \\
			&= \left\langle \operatorname{Res}\big((T_G-2)\otimes  \operatorname{Ind}\tau \big), \sigma \right\rangle_H= \langle (T_G-2) \otimes \operatorname{Ind}\tau,  \operatorname{Ind}\sigma  \rangle_G.
		\end{split}\]
	\end{proof}

\end{document}